\documentclass[12pt,a4paper]{amsart}
\usepackage{amsmath,amsfonts,amssymb,latexsym,amsthm}
\usepackage{hyperref}

\usepackage[english]{babel}
\usepackage[latin1]{inputenc}
\usepackage{times}
\usepackage[T1]{fontenc}

\usepackage[pdftex]{graphicx}

\newtheorem{theorem}{Theorem}
\newtheorem{proposition}{Proposition}

\newcommand{\Oo}{\mathcal{O}}
\newcommand{\CC}{\mathbb{C}}
\newcommand{\RR}{\mathbb{R}}
\newcommand{\QQ}{\mathbb{Q}}
\newcommand{\NN}{\mathbb{N}}
\newcommand{\ZZ}{\mathbb{Z}}
\newcommand{\Bo}{\mathcal{B}}
\newcommand{\La}{\mathcal{L}}

\everymath{\displaystyle}

\begin{document}
\keywords{hyperasymptotic expansions, heat equation, linear PDEs with constant coefficients, $k$-summability}
\subjclass[2010]{35C20, 35G10, 35K05}
\title{Hyperasymptotic solutions for certain partial differential equations}
\author{S{\l}awomir Michalik}
\address{Faculty of Mathematics and Natural Sciences,
College of Science\\
Cardinal Stefan Wyszy\'nski University\\
W\'oycickiego 1/3,
01-938 Warszawa, Poland}
\email{s.michalik@uksw.edu.pl}
\urladdr{\url{http://www.impan.pl/~slawek}}
\author{Maria Suwi\'nska}
\address{Faculty of Mathematics and Natural Sciences,
College of Science\\
Cardinal Stefan Wyszy\'nski University\\
W\'oycickiego 1/3,
01-938 Warszawa, Poland}
\email{m.suwinska@op.pl}

\begin{abstract}
We present the hyperasymptotic expansions for a certain group of solutions of the heat equation.
We extend this result to  a more general case of linear PDEs with constant coefficients. The generalisation is based on the method
of Borel summability, which allows us to find integral representations of solutions for such PDEs.
\end{abstract}

\maketitle

\section{Introduction}
Errors generated in the process of estimating functions by a finite number of terms of their asymptotic expansions usually are of the form 
$\exp(-q/t)$ with $t\to 0$ and usually such a result is satisfactory.
However, it is possible to obtain a refined information by means of finding the hyperasymptotic expansion of a given function, which amounts to expanding remainders of asymptotic expansions repeatedly.\par
More precisely, let us find the asymptotic expansion of a given function $F$. We receive
\begin{equation}
F(t)= A_0+A_1+\ldots\quad\textrm{for}\quad t\to 0\label{asympt_exp}
\end{equation}
with $A_i=a_it^i$. Once we truncate (\ref{asympt_exp}) after a certain amount of terms, we receive an approximation of $F$ and
$$
F(t)=A_0+A_1+\ldots+A_{N_0-1}+R_{N_0}(t).
$$
\par
The optimal value of $N_0=N_0(t)$ can be found by means of minimization of the remainder $R_{N_0}(t)$. After that we consider $R_{N_0}(t)$ as a function of two variables $t$ and $N_0$ and expand it in a new asymptotic series
$$
R_{N_0}(t)=B_0+B_1+\ldots,
$$
which can be truncated optimally after $N_1$ terms. Thus we receive an estimation of $F$ of the form
$$
F(t)=A_0+A_1\ldots+A_{N_0-1}+B_0+B_1+\ldots+B_{N_1-1}+R_{N_1}(t)
$$
and the remainder $R_{N_1}(t)$ appears to be exponentially small compared to $R_{N_0}(t)$.
\par
After repeating the process $n$ times we receive the $n$\textit{-th level hyperasymptotic expansion} of $F$ as $t\to 0$:
$$
F(t)=A_0+\ldots+A_{N_0-1}+B_0+\ldots+B_{N_1-1}+C_0+\ldots C_{N_2-1}+\ldots+R_{N_n}(t).
$$
\par
The concept of hyperasymptotic expansions emerged in 1990 as a topic of an article by M. V. Berry and C. J. Howls \cite{BH}
and it was conceived as a way to estimate the solutions of Schr\"odinger-type equations. Methods of obtaining hyperasymptotic expansions were then
developed mostly by A. B. Olde Daalhuis, who found an expansion for the confluent hypergeometric function \cite{O1a, O1b},
linear ODEs with the singularity of rank one \cite{O2} and various nonlinear ODEs \cite{O3a,O3b}.
\par
Using the results from \cite{O1a} and \cite{O1b}, we will find a hyperasymptotic expansion for a certain group of solutions of the heat equation.
To this end we will first obtain the optimal number of terms, after which the asymptotic expansion of the solution should be truncated.
This will enable us to estimate the remainder using the Laplace method (see \cite{Olv}). The reasoning then will be adapted to the case of
$n$-level hyperasymptotic expansion.
\par
Our main goal is to generalise those results to the case of linear PDEs with constant coefficients. To this end, first we
reduce the general linear PDEs in two variables with constant coefficients to simple pseudodifferential equations using the methods of
\cite{Mic7,Mic8}. Next, we apply the theory of summability, which allows to construct integral representations of solutions of such equations.
Finally, in a similar way to the heat equation, we construct hyperasymptotic expansions for such integral representations of solutions.
\bigskip
\par
Throughout the paper the following notation will be used.

A~sector $S$ in a direction $d\in\RR$ with an opening $\alpha>0$ in the universal covering space $\tilde{\CC}$ of
$\CC\setminus\{0\}$ is defined by $S_d(\alpha):=\{z\in\tilde{\CC}\colon z=r\*e^{i\varphi},\ r>0,\ \varphi\in(d-\alpha/2,d+\alpha/2)\}$.
If the opening $\alpha$ is not essential, the sector $S_d(\alpha)$ is denoted briefly by $S_d$.

We denote by $D_r$ a complex disc in $\CC$ with radius $r>0$ and the center in $0$, i.e. $D_r:=\{z\in\CC:\,|z|<r\}$.
In case that the radius $r$ is not essential, the set $D_r$ will be designated briefly by $D$.

If a function $f$ is holomorphic on a domain $G\subset\CC^n$, then it will be denoted by $f\in\Oo(G)$.
 Analogously, the space of holomorphic functions on a domain $G\subset\CC^n$ with respect to the variable
 $z^{1/\gamma}:=z_1^{1/\gamma_1},\dots,z_n^{1/\gamma_n}$,
 where $1/\gamma:=(1/\gamma_1,\dots,1/\gamma_n)$ and $(\gamma_1,\dots,\gamma_n)\in\NN^n$,
 is denoted by $\Oo_{1/\gamma}(G)$.
 
 By $\partial G$ we mean the boundary of the set $G$.

\section{Hyperasymptotic expansions for the heat equation}
Let us consider the Cauchy problem for the heat equation
\begin{equation}
\left\{ \begin{array}{rcl}
u_{t}(t,z)-u_{zz}(t,z) & = & 0\\
u(0,z) & = & \varphi(z).
\end{array}\right.\label{heat}
\end{equation}
We assume that the function $\varphi$ has finitely many isolated singular points (single-valued and branching points) on $\CC$. Without loss of
generality we may assume that the set of singular points of $\varphi$ is given by 
\begin{multline*}
\mathcal{A}:=\{a_{ij}\in\CC\colon \arg(a_{i1})=\dots=\arg(a_{iL_i})=\lambda_i,\ |a_{i1}|<|a_{i2}|<\ldots<|a_{iL_i}|,\\ j=1,\dots,L_i,\ i=1,\dots,K\},
\end{multline*}
where $K\in\NN$, $L_1,\dots,L_K\in\NN$ and $\lambda_1,\dots,\lambda_K\in\RR$ satisfy $\lambda_1<\dots<\lambda_K$.

Under these conditions we can define the set $H$ as a sum of a finite number of half-lines (see Figure \ref{figure1})
such that $H:=\bigcup_{i=1}^K\{a_{i1}t:\,t\geq 1\}$. So we may assume that $\varphi\in\Oo(\CC\setminus H)$ and $\mathcal{A}$ is
the set of all singular points of $\varphi$.
We denote it briefly by $\varphi\in\Oo_{\mathcal{A}}(\CC\setminus H)$.

Moreover, let us assume that for any $\xi>0$ there exist positive constants $B$ and $C$ such that $|\varphi(z)|\leq Ce^{B|z|^{2}}$ for all
$z\in\CC\setminus H_\xi$, where $H_\xi:=\{z\in\CC\colon \textrm{dist\,}(z,H)<\xi\}$.
We write it $\varphi\in\Oo^2_{\mathcal{A}}(\CC\setminus H)$ for short.
\begin{figure}[htb]
\centering{}\includegraphics[width=0.7\textwidth]{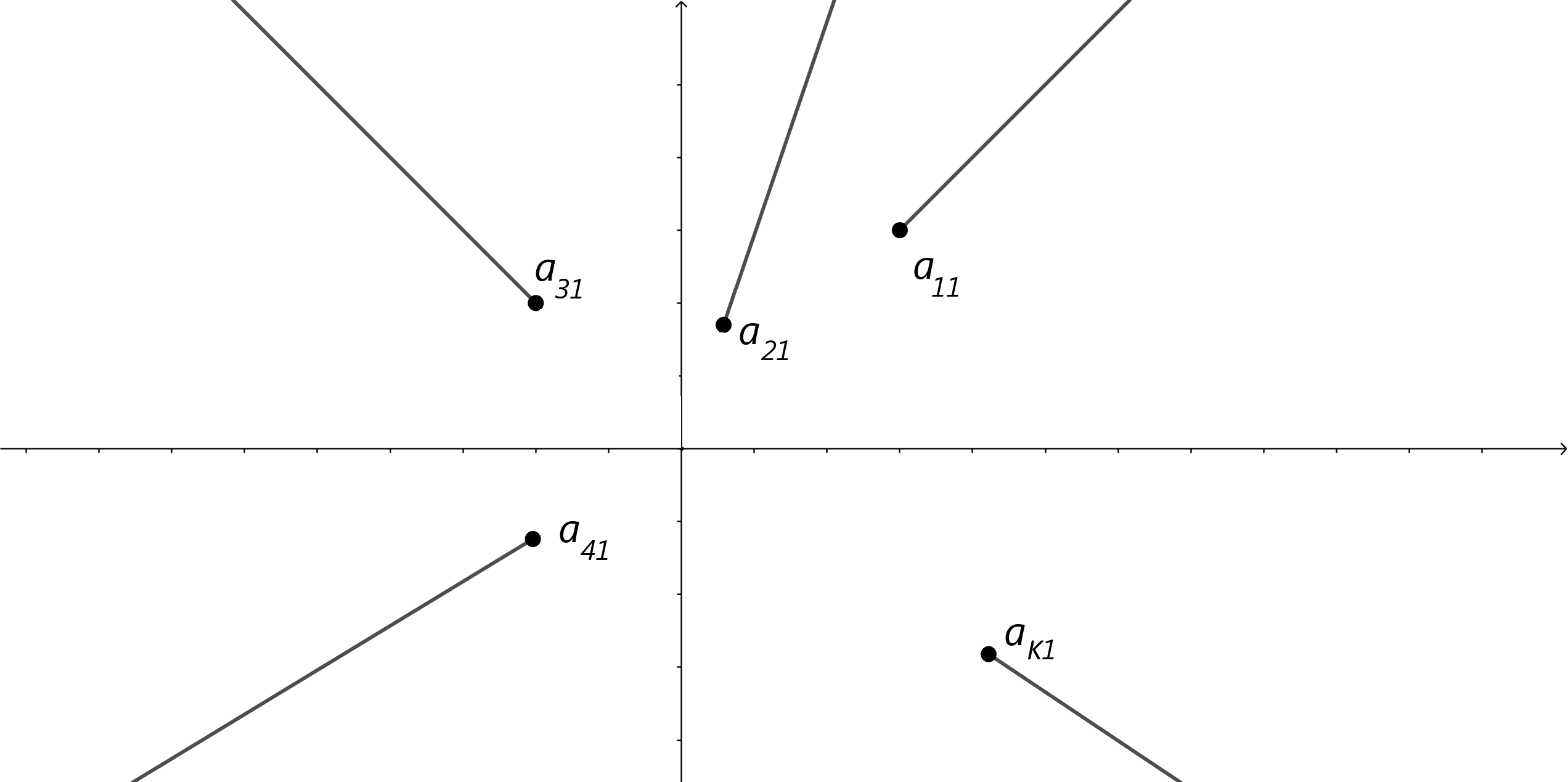}\caption{\label{figure1}}
\end{figure}
\par
The solution of (\ref{heat}) is given by (see \cite[Theorem 4]{Mic-Pod})
\begin{equation}
u(t,z)=\frac{1}{2\sqrt{\pi t}}\int_{e^{i\frac{\theta}{2}}\mathbb{R}}e^{-\frac{s^{2}}{4t}}\varphi(z+s)\,ds\label{heat_solution}
\end{equation}
under condition that $\theta$ is not the Stokes line for $u$ (see \cite[Definition 7]{Mic-Pod},
i.e. $\theta\neq 2\lambda_i\mod 2\pi$ for $i=1,\dots,K$.

To separate from the Stokes lines we fix a small positive number $\delta$ and we assume that
\begin{gather}
 \label{eq:delta}
 |(\theta-2\lambda_i)\mod 2\pi|\geq\delta\quad\textrm{for all}\quad i=1,\dots,K.
\end{gather}
 In other words we assume that 
$\theta\in[0,\,2\pi)\setminus\bigcup_{i=1}^{K}(2\lambda_i-\delta,\,2\lambda_i+\delta)\mod 2\pi$.

Our goal is to find a hyperasymptotic expansion of (\ref{heat_solution}) for $t\to0$
and $\arg t=\theta$ with $z$ belonging to a small neighbourhood
of $0$. To this end we fix a sufficiently small constant $\tilde{\varepsilon}$ such that $\varphi(z)\in\Oo(D_{\tilde{\varepsilon}})$.

\subsection{0-level hyperasymptotic expansion}\label{O_level_exp}
To find the hyperasymptotic expansion of the solution of (\ref{heat}) we will use the method described in \cite{O1a} (see also \cite{O1b})
in the case of the confluent hypergeometric functions. In order to do so let us modify the right-hand side of (\ref{heat_solution}) by
\[
u(t,z)=\frac{1}{4\sqrt{\pi t}}\int_{0}^{e^{i\theta}\infty}e^{-\frac{s}{4t}}s^{-\frac{1}{2}}\left[\varphi\left(z+s^{\frac{1}{2}}\right)+\varphi\left(z-s^{\frac{1}{2}}\right)\right]\,ds.
\]
Replacing $s$ and $t$ by $|s|e^{i\theta}$ and
$|t|e^{i\theta}$, respectively, we obtain
\begin{equation}
u(t,z)=\frac{1}{2\sqrt{\pi|t|}}\int_{0}^{\infty}\frac{e^{-\frac{s}{4|t|}}}{2\sqrt{s}}
\left[\varphi\left(z+e^{i\frac{\theta}{2}}\sqrt{s}\right)+\varphi\left(z-e^{i\frac{\theta}{2}}\sqrt{s}\right)\right]\,ds.\label{heat_subst}
\end{equation}
To find the asymptotic expansion of (\ref{heat_subst}), we will expand the
function
$$f_{0}(s,z):=\frac{1}{2}\left[\varphi\left(z+e^{i\frac{\theta}{2}}\sqrt{s}\right)+\varphi\left(z-e^{i\frac{\theta}{2}}\sqrt{s}\right)\right]$$
around the point $s=0$ using the complex Taylor formula. We receive
\begin{equation}
f_{0}(s,z)=\sum_{k=0}^{N_{0}-1}\frac{\varphi^{\left(2k\right)}\left(z\right)}{\left(2k\right)!}e^{ik\theta}s^{k}+f_{1}(s,z)s^{N_0},\label{f_0_expansion}
\end{equation}
where $f_{1}(s,z)$ is of the form
\begin{equation}
f_{1}(s,z):=\frac{1}{2\pi i}\int_{\Omega_{0}\left(0,\,s\right)}\frac{f_{0}(w,z)}{w^{N_{0}}\left(w-s\right)}\,dw\label{f_0_Taylor}
\end{equation}
and the contour $\Omega_{0}(0,s)$ is a boundary of the sum of two
discs such that all singular points of $f_{0}(w,z)$ are located outside
of those discs and points $0$ and $s$ are both inside. More precisely,
let us take $r:=\min_{1\leq i\leq K}{|a_{i1}|}-\tilde{\varepsilon}.$ In
this case we can put $\Omega_{0}(0,s)$ as
\[
\Omega_{0}(0,s):=\partial\Big( \{ w\in\CC\colon |w|\leq r^{2}-\varepsilon\}\cup\{ w\in\CC\colon |w-s|\leq\frac{\varepsilon}{2}\}\Big) 
\]
for some $\varepsilon\in (0,r^{2}/2)$ and $\varepsilon$ separate from $0$. It is possible to take
such a contour, because by (\ref{eq:delta}) we may choose so small $\tilde{\varepsilon}>0$ that for $z\in D_{\tilde{\varepsilon}}$ 
the singularities $w_{ij}(z):=(a_{ij}-z)^{2}e^{-i\theta}$ of $f_{0}(w,z)$ will never be positive real numbers. 
So we are able to choose $\varepsilon$ satisfying additionally
$$
\varepsilon< \frac{1}{2}\inf_{z\in D_{\tilde{\varepsilon}}}\min_{i=1,\dots K\atop j=1,\dots,L_i}\textrm{dist\,}(w_{ij}(z),\RR_+)
$$
and then $\Omega_{0}(0,s)$ satisfies the desired conditions.
\par
Using (\ref{f_0_expansion}) and basic properties of the gamma function,
we can obtain an expansion of (\ref{heat_subst}) of the form
\begin{equation}
u(t,z)=\sum_{k=0}^{N_{0}-1}\frac{\varphi^{\left(2k\right)}\left(z\right)}{k!}t^{k}+R_{N_{0}}\left(t,z\right),\label{0level}
\end{equation}
where
\begin{equation}
R_{N_{0}}\left(t,z\right)=\frac{1}{2\sqrt{\pi|t|}}\int_{0}^{\infty}e^{-\frac{s}{4|t|}}s^{N_{0}-\frac{1}{2}}f_{1}\left(s,z\right)\,ds.\label{0level_rem}
\end{equation}
\par
Seeing as $|w|\geq r^{2}-\varepsilon$, $|w-s|\geq\frac{\varepsilon}{2}$
and assuming that all the conditions given for the Cauchy datum hold, we
can find the optimal value of $N_{0}=N_{0}(t)$. The first step to
do so is finding an estimation of $f_{1}(s,z).$ Let us note that there exist positive constants $\tilde{A}$ and $\tilde{B}$
such that 
$\left|\varphi\left(z\pm e^{i\frac{\theta}{2}}\sqrt{w}\right)\right|\leq\tilde{A}e^{\tilde{B}s}$
for any $w\in\Omega_{0}(0,s)$, $s>0$ and $z\in D_{\tilde{\varepsilon}}$.  Hence
\begin{multline*}
|f_{1}(s,z)|\leq \frac{1}{2\pi}\int_{\Omega_{0}\left(0,\,s\right)}\frac{2\tilde{A}e^{\tilde{B}s}}{\varepsilon(r^2-\varepsilon)^{N_{0}}}\,d|w|
\leq \frac{2\tilde{A}e^{\tilde{B}s}}{\varepsilon(r^2-\varepsilon)^{N_{0}}}(r^2-\varepsilon+\frac{\varepsilon}{2})\\
\leq
\frac{2\tilde{A}r^{2}e^{\tilde{B}s}}{\varepsilon(r^{2}-\varepsilon)^{N_{0}}}=\frac{A_{0}e^{\tilde{B}s}}{(r^{2}-\varepsilon)^{N_{0}}}
\end{multline*}
for $A_0:=2\tilde{A}r^2/\varepsilon$. As a consequence,
\begin{equation}
|R_{N_{0}}\left(t,z\right)|\leq\frac{A_{0}}{2\sqrt{\pi|t|}}\int_{0}^{\infty}e^{-\frac{s}{4|t|}}s^{N_{0}-\frac{1}{2}}\left(r^{2}-\varepsilon\right)^{-N_{0}}e^{\tilde{B}s}\,ds.\label{RN0-est}
\end{equation}
It is easy to check that the integrand of (\ref{RN0-est}) has a maximum
at a certain point $s=\sigma_{1}$ which satisfies the condition $N_{0}=\sigma_{1}\left(\frac{1}{4|t|}-\tilde{B}\right)+\frac{1}{2},$
and
 so now we can find the point where the minimum with respect to $\sigma_{1}$
of the function given by the formula
\[
\sigma_{1}\mapsto e^{-\sigma_{1}\left(\frac{1}{4|t|}-\tilde{B}\right)}\sigma_{1}^{\sigma_{1}\left(\frac{1}{4|t|}-\tilde{B}\right)}\left(r^{2}-\varepsilon\right)^{-\sigma_{1}\left(\frac{1}{4|t|}-\tilde{B}\right)-\frac{1}{2}}
\]
is attained. This function is minimal at $\sigma_{1}=r^{2}-\varepsilon$.
Because of these facts we can choose the optimal
$N_{0}:=\left\lfloor(r^{2}-\varepsilon)\left(\frac{1}{4|t|}-\tilde{B}\right)+\frac{1}{2}\right\rfloor$,
where by $\lfloor\cdot\rfloor$ we denote the integer part of a real number. Next, we take
$\sigma_1:=\frac{N_0-\frac{1}{2}}{\frac{1}{4|t|}-\tilde{B}}$. Thus $\sigma_1\leq r^2-\varepsilon$.
\par
Thanks to that we are able to use the Laplace method, described at length in \cite{Olv}, and to estimate the right-hand side of (\ref{RN0-est}).
So, we conclude that
\begin{gather*}
\left|R_{N_{0}}\left(t,\,z\right)\right|\sim O\left(\frac{e^{-\sigma_{1}\left(\frac{1}{4|t|}-\tilde{B}\right)}}{\sqrt{1-4\tilde{B}|t|}}\right)
\quad\textrm{for}\quad
t\to 0,\quad \arg t=\theta.
\end{gather*}

\subsection{n-level hyperasymptotic expansion}
When the $n$-level asymptotic expansion is known, it is easy to compute
the ($n+1$)-level expansion using the method presented in Section \ref{O_level_exp}.
\par
Observe that the remainder obtained in the $n$-level hyperasymptotic
expansion is of the form
\begin{equation}
{\displaystyle R_{N_{n}}(t,\,z)=\frac{1}{2\sqrt{\pi|t|}}\int_{0}^{\infty}e^{-\frac{s}{4|t|}}s^{N_{0}-\frac{1}{2}}(s-\sigma_{1})^{N_{1}}\cdot\ldots\cdot(s-\sigma_{n})^{N_{n}}f_{n+1}(s,z)\,ds}\label{RNn}
\end{equation}
where 
\[
f_{n+1}(s,z):=\frac{1}{2\pi i}\int_{\Omega_{n}\left(\sigma_{n},\,s\right)}\frac{f_{n}(w,z)}{(w-\sigma_{n})^{N_{n}}\left(w-s\right)}\,dw
\]
and 
$$\Omega_{n}(\sigma_{n},s):=
\partial\Big(\{ w\in\CC\colon |w-\sigma_{n}|\leq d(\sigma_{n},\theta)-\rho_{n}\varepsilon\}\cup
\{ w\in\CC\colon |w-s|\leq 2^{-n-1}\varepsilon\}\Big)
$$
with $d(w,\theta):=\inf_{z\in D_{\tilde{\varepsilon}}}\inf_{\zeta\in H}|w-e^{-i\theta}(z-\zeta)^{2}|$
and $\rho_{n}:=2-2^{-n}$. The contour is chosen in this way so that,
when we express $f_{n+1}$ in terms of $f_{0}$, that is as a multiple
integral of the form
$$
\begin{array}{rcl}
\displaystyle f_{n+1}(s,z)&=&\frac{1}{(2\pi i)^{n+1}}\int_{\Omega_{n}\left(\sigma_{n},s\right)}\int_{\Omega_{n-1}\left(\sigma_{n-1},\,x_{n}\right)}\ldots\int_{\Omega_{0}\left(0,\,x_{1}\right)}f_{0}(x_{0},z)\cdot\\
&&\displaystyle\cdot\frac{1}{{x_0^{N_0}\prod_{k=1}^{n}\left[\left(x_{k}-\sigma_{k}\right)^{N_{k}}\left(x_{k-1}-x_{k}\right)\right]\left(x_{n}-s\right)}}\,dx_{0}\ldots dx_{n}.
\end{array}
$$
We show that all the singular points of $x_0\mapsto f_0(x_0,z)$ are outside of the area surrounded by $\Omega_n(\sigma_n,s)$,
$\Omega_{n-1}(\sigma_{n-1},x_n)$, ..., $\Omega_0(0,x_1)$. To this end we take $x_k\in\Omega_k(\sigma_k,x_{k+1})$ for $k=0,\dots,n$ with the notation
$\sigma_0:=0$ and $x_{n+1}:=s$. It is sufficient to prove that $d(x_0,\theta)\geq\varepsilon$. There are two possibilities.

In the first case
$|x_k-x_{k-1}|=2^{-k-1}\varepsilon$ for $k=1,\dots,n$. Then
\[
|x_{0}-s|\leq\sum_{k=0}^{n}|x_{k+1}-x_{k}|\leq\sum_{k=0}^{n}2^{-k-1}\varepsilon=(1-2^{-n-1})\varepsilon.
\]
Since $d(s,\theta)\geq 2\varepsilon$ we get $d(x_0,\theta)\geq d(s,\theta)-|x_0-s|\geq\varepsilon$.

In the second case there exists $m\in\{1,\dots,n\}$ such that $|x_k-x_{k-1}|=2^{-k-1}\varepsilon$ for $k=1,\dots,m-1$ and
$|x_m-\sigma_m|=d(\sigma_m,\theta)-\rho_{m}\varepsilon$. Hence $|x_0-x_m|\leq (1-2^{-m})\varepsilon$ and
$d(x_m,\theta)\geq d(\sigma_m,\theta)-|x_m-\sigma_m|\geq \rho_{m}\varepsilon$, so we conclude that
$d(x_0,\theta)\geq d(x_m,\theta)-|x_0-x_m|\geq\varepsilon$.
\par
Using the same algorithm as in the case of the $0$-level expansion, we
can estimate $R_{N_{n}}(t,z)$ as follows
\begin{eqnarray}
|R_{N_{n}}(t,z)|&\leq&\frac{A_{n}}{2\sqrt{\pi|t|}}\int_{0}^{\infty}e^{-\frac{s}{4|t|}+\tilde{B}s}\frac{s^{N_{0}-\frac{1}{2}}}{\left(r^{2}-\varepsilon\right)^{N_{0}}}\cdot\frac{|s-\sigma_{1}|^{N_{1}}}{\left(d(\sigma_{1},\theta)-\rho_{1}\varepsilon\right)^{N_{1}}}\cdot\ldots\cdot\nonumber\\
&&\qquad\cdot\frac{|s-\sigma_{n}|^{N_{n}}}{\left(d(\sigma_{n},\theta)-\rho_{n}\varepsilon\right)^{N_{n}}}\,ds\label{RNn_est}
\end{eqnarray}
for a certain constant $A_{n}$.
\par
Next, we find points where the integrand on the right-hand side of
(\ref{RNn_est}) attains its local maxima. Let us observe that this
function has $n+1$ maxima in points $s_{1},\ldots,\,s_{n+1}$ such
that $s_{1}<\sigma_{1}<s_{2}<\ldots<\sigma_{n}<s_{n+1}$ and all $s_{j}$
satisfy the condition:
\begin{equation}
\frac{1}{4|t|}-\tilde{B}=\frac{N_{0}-\frac{1}{2}}{s_{j}}+\frac{N_{1}}{s_{j}-\sigma_{1}}+\ldots+\frac{N_{n}}{s_{j}-\sigma_{n}}.\label{cond}
\end{equation}
From (\ref{cond}) we conclude that  $s_{j}$ are decreasing functions of $N_{n}$
for $1\leq j\leq n$ and $s_{n+1}$ increases with respect to $N_{n}$.
Moreover, the value of the integrand in (\ref{RNn_est}) in the points $s_j$, $1\leq j\leq n$, decreases with respect to $N_{n}$. However, it behaves differently in the point $s_{n+1}$. It decreases with respect to $N_{n}$
when 
\[
{\displaystyle {\displaystyle \frac{s_{n+1}-\sigma_{n}}{d(\sigma_{n},\theta)-\rho_{n}\varepsilon}}}<1,
\]
 that is for $s_{n+1}\in(\sigma_{n},\:\sigma_{n}+d(\sigma_{n},\theta)-\rho_{n}\varepsilon)$,
and increases when $s_{n+1}\in(\sigma_{n}+d(\sigma_{n},\theta)-\rho_{n}\varepsilon,\,+\infty)$.
Hence there exist $N_{n}\in\mathbb{N}$ and $s_{n+1}=:\sigma_{n+1}$ satisfying (\ref{cond}) for which the integrand reaches
its minimal value (see \cite{O1b}).
\par
Again, we can use the Laplace method to obtain the estimation of $R_{N_n}$ (compare \cite{O1a} and \cite{O1b})
\begin{gather*}
|R_{N_{n}}(t,z)|\sim O\left(\frac{e^{-\eta_{n}\left(\frac{1}{4|t|}-\tilde{B}\right)}}{\sqrt{1-4\tilde{B}|t|}}\right)
\quad\textrm{for}\quad t\to 0,\quad \arg t=\theta.
\end{gather*}
We have the sequence of positive numbers $\eta_0=\sigma_1\sim r^2<\eta_1<\eta_2<\eta_3<\dots$, but it is not clear, whether or not,
$\{\eta_n\}_{n\in\NN}$
is an unbounded sequence (see \cite{O1a} and \cite{O1b}). 
\par
To find the $(n+1)$-level hyperasymptotic expansion we expand the function $s\mapsto f_{n+1}(s,z)$
around the point $\sigma_{n+1}$. As a result we receive a
series
\begin{equation}
\label{eq:f_n1}
f_{n+1}(s,z)=\sum_{j=0}^{N_{n+1}-1}b_{n+1,j}(z)(s-\sigma_{n+1})^{j}+(s-\sigma_{n+1})^{N_{n+1}}f_{n+2}(s,z),
\end{equation}
which, after substituting it in (\ref{RNn}), gives us the ($n+1$)-level
expansion of the form
\begin{eqnarray*}
R_{N_{n}}(t,z)&=&\frac{1}{\sqrt{\pi|t|}}\sum_{j=0}^{N_{n+1}-1}b_{n+1,j}(z)\int_{0}^{\infty}\frac{e^{-\frac{s}{4|t|}}s^{N_{0}}}{2\sqrt{s}}
(s-\sigma_{1})^{N_{1}}\cdot\ldots\cdot\nonumber\\
&&\qquad\qquad\qquad\qquad\cdot(s-\sigma_{n})^{N_{n}}(s-\sigma_{n+1})^{j}ds+R_{N_{n+1}}(t,z).
\end{eqnarray*}
Moreover, since 
\begin{equation}
\label{eq:polynomial}
s^{N_{0}}(s-\sigma_{1})^{N_{1}}\cdots(s-\sigma_{n})^{N_{n}}(s-\sigma_{n+1})^{j}=\sum_{l=0}^{N_0+\cdots+N_n+j}a_{n,j,l}s^l
\end{equation}
is a polynomial of degree $N_0+\cdots+N_{n}+j$, and by the properties of the gamma function
\begin{gather*}
 \frac{1}{\sqrt{\pi|t|}}\int_0^{\infty}\frac{e^{-\frac{s}{4|t|}}}{2\sqrt{s}}s^{l}\,ds=\frac{(2l)!}{l!}|t|^l
 =\frac{(2l)!}{l!}e^{-i\theta l}t^l  \quad\textrm{for}\quad l=0,1,\dots,
\end{gather*}
we conclude that
$$
R_{N_{n}}(t,z)=\sum_{j=0}^{N_{n+1}-1}b_{n+1,j}(z)\sum_{l=0}^{N_0+\cdots+N_n+j}\frac{(2l)!}{l!}a_{n,j,l}e^{-i\theta l}t^l+
R_{N_{n+1}}(t,z).
$$
Hence the hyperasymptotic expansion of $u$ takes the form
\begin{equation}
\label{eq:exp_1}
u(t,z)=\sum_{l=0}^{N_0+\cdots+N_n-1}\psi_l(z)t^l+R_{N_{n}}(t,z)
\end{equation}
for some functions $\psi_l(z)$ depending on $b_{n+1,j}(z)$ and $a_{n,j,l}$.

\subsection{Conclusion}
We can formulate the following theorem regarding the hyperasymptotic expansion of (\ref{heat})
\begin{theorem}
For any $n\in\mathbb{N}$ the solution (\ref{heat_solution}) of the heat equation has the hyperasymptotic expansion as $t\to 0$ in the direction
$\theta\in[0,\,2\pi)\setminus\bigcup_{i=1}^{K}(2\lambda_i-\delta,\,2\lambda_i+\delta) \mod 2\pi$ of the form
\begin{multline*}
u(t,z)=\sum_{j=0}^{N_{0}-1}\frac{\varphi^{\left(2j\right)}\left(z\right)}{j!}t^{j}+
\sum_{m=1}^{n}\sum_{j=0}^{N_{m}-1}\frac{b_{m,j}(z)}{\sqrt{\pi|t|}}\int_{0}^{\infty}\frac{e^{-\frac{s}{4|t|}}s^{N_{0}}}{2\sqrt{s}}
(s-\sigma_{1})^{N_{1}}\\
\cdots(s-\sigma_{m-1})^{N_{m-1}}(s-\sigma_{m})^{j}ds+R_{N_{n}}(t,z)
=\sum_{l=0}^{N_0+\cdots+N_n-1}\psi_l(z)t^l+R_{N_{n}}(t,z),
\end{multline*}
where the remainder $R_{N_{n}}(t,z)$ is of the form
\begin{equation*}
 R_{N_{n}}(t,\,z)=\frac{1}{2\sqrt{\pi|t|}}\int_{0}^{\infty}e^{-\frac{s}{4|t|}}s^{N_{0}-\frac{1}{2}}(s-\sigma_{1})^{N_{1}}
 \cdot\ldots\cdot(s-\sigma_{n})^{N_{n}}f_{n+1}(s,z)\,ds
\end{equation*}
and for any $m\leq n$ and $j<N_{m}$
\begin{multline*}
f_m(s,z)=\frac{1}{(2\pi i)^m}\int_{\Omega_{m-1}(\sigma_{m-1},s)}\int_{\Omega_{m-2}(\sigma_{m-2},x_{m-1})}\cdots
 \int_{\Omega_{1}(\sigma_{1},x_2)}\int_{\Omega_{0}(0,x_1)}\\
   \frac{f_{0}(x_{0},z)\,dx_0\dots dx_{m-1}}{{x_0^{N_0}
   \displaystyle\prod_{k=1}^{m-1}\left[\left(x_{k}-\sigma_{k}\right)^{N_{k}}\left(x_{k-1}-x_{k}\right)\right]\left(x_{m-1}-s\right)}}
\end{multline*}
and
$$
b_{m,j}(z)=\frac{1}{j!}\frac{\partial^{j}}{\partial s^{j}}f_{m}(s,z)\left|_{s=\sigma_{m}}.\right.
$$
\par
Moreover,
$$
|R_{N_{n}}(t,z)|\sim O\left(\frac{e^{-\eta_{n}\left(\frac{1}{4|t|}-\tilde{B}\right)}}{\sqrt{1-4\tilde{B}|t|}}\right)\quad\textrm{as}\quad
t\to 0,\ \arg t=\theta,\ z\in D_{\tilde{\varepsilon}}
$$
for some sequence of positive numbers $\eta_0=\sigma_1\sim r^2<\eta_1<\eta_2<\eta_3<\dots$.
\end{theorem}

\section{Generalisation to linear PDEs with constant coefficients}
In this section we show how to find the hyperasymptotic expansion for solutions of general linear non-Cauchy-Kowalevskaya type
PDEs with constant coefficients. The result is based on the theory of summability which allows us to construct the actual solution,
which is analytic
in some sectorial neighbourhood of the origin, from the divergent formal power series solution. Moreover this actual solution has an integral
representation in the similar form to (\ref{heat_solution}).

\subsection{Summability}
First, we define $k$-summability in a similar way to \cite{Mic-Pod}.
For more information about the theory of summability we refer the reader to \cite{Balser}.

We say that a formal power series $\hat{u}(t,z)=\sum_{n=0}^{\infty}\frac{u_n(z)}{n!}t^n$ with $u_n(z)\in\Oo_{1/\kappa}(D)$
is a \emph{Gevrey series} of order $q$ if there exist
$A,B,r>0$ such that $|u_n(z)|\leq A B^n(n!)^{q+1}$ for every $|z|<r$ and every $n\in\NN$.
We denote by $\Oo_{1/\kappa}(D)[[t]]_q$ the set of such formal power series.
\par
Moreover, for $k>0$ and $d\in\RR$, we say that $\hat{u}(t,z)\in \Oo_{1/\kappa}(D)[[t]]_{\frac{1}{k}}$ is \emph{$k$-summable in a direction $d$}
if its $k$-Borel transform
$$v(s,z):=(\Bo_k\hat{u})(s,z):=\sum_{n=0}^{\infty}\frac{u_n(z)}{\Gamma(1+(1+1/k)n)}s^n,$$
where $\Gamma(\cdot)$ denotes the gamma function,
is analytically
continued with respect to $s$ to an unbounded sector $S_d$ in a direction $d$ and this analytic continuation has exponential growth of order $k$
as $s$ tends to infinity (i.e. $|v(s,z)|\leq Ae^{B|s|^k}$ as $s\to\infty$). We denote it briefly by
$v(s,z)\in\Oo_{1,1/\kappa}^k((D\cup S_d)\times D)$. In this case the \emph{$k$-sum of $\hat{u}(t,z)$ in the direction $d$}
is given by
$$
u^d(t,z):=(\La_{k,d}v)(t,z):=t^{-k/(1+k)}\int_{e^{id}\RR_+}v(s,z)\,C_{(k+1)/k}((s/t)^{\frac{k}{1+k}})\,ds^{\frac{k}{1+k}},
$$
where $C_{\alpha}(\tau)$ is the \emph{Ecalle kernel} defined by
 $$
C_{\alpha}(\tau):=\sum_{n=0}^{\infty}\frac{(-\tau)^n}{n!\,\*\Gamma\bigl(1-\frac{n+1}{\alpha}\bigr)}.
$$

\subsection{Reduction of linear PDEs with constant coefficients to simple pseudodifferential equations}
We consider the Cauchy problem
\begin{equation}
 \label{eq:pde}
 \begin{cases}
  P(\partial_{t},\partial_{z})u=0&\\
  \partial_{t}^j u(0,z)=\varphi_j(z)\in\Oo_{\mathcal{A}}(\CC\setminus H),
 \end{cases}
\end{equation}
where
$
P(\lambda,\zeta):=P_0(\zeta)\lambda^m-\sum_{j=1}^m P_j(\zeta)\lambda^{m-j}
$
is a general polynomial of two variables, which is of
order $m$ with respect to $\lambda$.
\par
First, we show how to use the methods from \cite{Mic7,Mic8} for the reduction of (\ref{eq:pde}) to simple pseudodifferential equations.

If  $P_0(\zeta)$ is not a constant, then a formal solution of (\ref{eq:pde})
    is not uniquely determined. To avoid this inconvenience we choose some special solution which is already 
    uniquely determined. To this end we factorise the polynomial $P(\lambda,\zeta)$ as follows
   \begin{equation}
   \label{eq:factorise}
    P(\lambda,\zeta)=P_0(\zeta)(\lambda-\lambda_1(\zeta))^{m_1}\cdots(\lambda-\lambda_l(\zeta))^{m_l}
    =:P_0(\zeta)\widetilde{P}(\lambda,\zeta),
    \end{equation}
   where $\lambda_1(\zeta),\dots,\lambda_l(\zeta)$ are the roots of the characteristic equation
   $P(\lambda,\zeta)=0$ with multiplicity
   $m_1,\dots,m_l$ ($m_1+\cdots+m_l=m$) respectively.
   \par
     Since $\lambda_{\alpha}(\zeta)$ are algebraic functions,
     we may assume that there exist $\kappa\in\NN$ and $r_0<\infty$ such that
   $\lambda_{\alpha}(\zeta)$ are holomorphic functions of the variable $\xi=\zeta^{1/\kappa}$
   (for $|\zeta|\geq r_0$ and $\alpha=1,\dots,l$) and, moreover, there exist $\lambda_{\alpha}\in\CC\setminus\{0\}$ and 
   $q_{\alpha}=\mu_{\alpha}/\nu_{\alpha}$
   (for some relatively prime numbers $\mu_{\alpha}\in\ZZ$ and $\nu_{\alpha}\in\NN$) such that
   $\lambda_{\alpha}(\zeta)\sim\lambda_{\alpha}\zeta^{q_{\alpha}}$ for $\alpha=1,\dots,l$ (i.e.
$\lim_{\zeta\to\infty}\frac{\lambda_{\alpha}(\zeta)}{\zeta^{q_{\alpha}}}=\lambda_{\alpha}$, $\lambda_{\alpha}$ and $q_{\alpha}$ are called
respectively a \emph{leading term} and a \emph{pole order} of $\lambda_{\alpha}(\zeta)$). Observe that $\nu_{\alpha}|\kappa$ for $\alpha=1,\dots,l$.

Following \cite[Definition 13]{Mic8} we define the \emph{pseudodifferential operators $\lambda_{\alpha}(\partial_{z})$} as
\begin{equation}
  \label{eq:lambda}
  \lambda_{\alpha}(\partial_{z})\varphi(z):=\frac{1}{2\kappa\pi i} \oint^{\kappa}_{|w|=\varepsilon}\varphi(w)
  \int_{e^{i\theta}r_0}^{e^{i\theta}\infty}\lambda_{\alpha}(\zeta)
  \mathbf{E}_{1/\kappa}(\zeta^{1/\kappa} z^{1/\kappa})e^{-\zeta w}\,d\zeta\,dw
 \end{equation}
for every $\varphi\in\Oo_{1/\kappa}(D_r)$ and $|z|<\varepsilon < r$, where
$\mathbf{E}_{1/\kappa}(t):=\sum_{n=0}^{\infty}\frac{t^n}{\Gamma(1+n/\kappa)}$
is the \emph{Mittag-Leffler function} of order $1/\kappa$, $\theta\in (-\arg w-\frac{\pi}{2}, -\arg w + \frac{\pi}{2})$ and
$\oint_{|w|=\varepsilon}^{\kappa}$ means that we integrate $\kappa$ times along the positively oriented circle of radius $\varepsilon$.
Here the integration in the inner integral is taken over the ray $\{e^{i\theta}r\colon r\geq r_0\}$.
\par
Under the above assumption, by a \emph{normalised formal solution} $\hat{u}$ of (\ref{eq:pde}) we mean such solution
   of (\ref{eq:pde}), which is also a solution of the pseudodifferential equation
   $\widetilde{P}(\partial_{t},\partial_{z})\hat{u}=0$ (see \cite[Definition 10]{Mic7}).
   \par
Since the principal part of the pseudodifferential operator $\widetilde{P}(\partial_t,\partial_z)$ with respect to $\partial_t$
is given by $\partial_t^m$,
the Cauchy problem (\ref{eq:pde}) has a unique normalised formal power series solution $\hat{u}\in\Oo(D)[[t]]$.

Next, we reduce the Cauchy problem (\ref{eq:pde}) of a general linear partial differential equation with constant coefficients to
a family of the Cauchy problems of simple pseudodifferential equations. Namely we have
\begin{proposition}[{\cite[Theorem 1]{Mic8}}]
\label{pr:family}
Let $\hat{u}$ be the normalised formal solution of (\ref{eq:pde}).
Then
$\hat{u}=\sum_{\alpha=1}^l\sum_{\beta=1}^{m_{\alpha}}\hat{u}_{\alpha\beta}$ with
   $\hat{u}_{\alpha\beta}$ being a formal solution of a simple pseudodifferential equation
   \begin{equation}
   \label{eq:gevrey}
    \left\{
    \begin{array}{l}
     (\partial_{t}-\lambda_{\alpha}(\partial_{z}))^{\beta} \hat{u}_{\alpha\beta}=0\\
     \partial_{t}^j \hat{u}_{\alpha\beta}(0,z)=0\ \ (j=0,\dots,\beta-2)\\
     \partial_{t}^{\beta-1} \hat{u}_{\alpha\beta}(0,z)=\lambda_{\alpha}^{\beta-1}(\partial_{z})
     \varphi_{\alpha\beta}(z),
    \end{array}
    \right.
   \end{equation}
   where $\varphi_{\alpha\beta}(z):=\sum_{j=0}^{m-1}d_{\alpha\beta j}(\partial_{z})
   \varphi_j(z)\in\Oo_{1/\kappa}(D)$ and $d_{\alpha\beta j}(\zeta)$ are some holomorphic
   functions of the variable $\xi=\zeta^{1/\kappa}$ and of polynomial growth. 
   \par
   Moreover, if $q_{\alpha}$ is a pole order of $\lambda_{\alpha}(\zeta)$
   and $\overline{q}_{\alpha}=\max\{0,q_{\alpha}\}$,
   then a formal solution $\hat{u}_{\alpha\beta}$ is a Gevrey series of order $\overline{q}_{\alpha} - 1$
   with respect to $t$.
\end{proposition}

For this reason we will study the following simple pseudodifferential equation 
\begin{equation}
 \label{eq:simple}
 \begin{cases}
  (\partial_t-\lambda(\partial_z))^{\beta}u=0\\
  \partial_t^j u(0,z)=0\ (j=0,\dots,\beta-2)\\
  \partial_t^{\beta-1}u(0,z)=\lambda^{\beta-1}(\partial_z)\varphi(z)\in\Oo_{1/\kappa}(D),
 \end{cases}
\end{equation}
where $\lambda(\zeta)\sim \lambda\zeta^q$ for some $q\in\QQ$, $q>1$. So we assume that $q=\mu/\nu$ for some relatively prime $\mu,\nu\in\NN$,
$\mu>\nu$.

\subsection{Summable solutions of simple pseudodifferential equations}
We have the following representation of summable solutions of (\ref{eq:simple}).
\begin{theorem}
\label{th:sum}
 Let $k:=(q-1)^{-1}$ and $d\in\RR$. Suppose that $\hat{u}(t,z)$ is a unique formal power series solution of the Cauchy problem (\ref{eq:simple})
 and 
 \begin{equation}
 \label{eq:condition}
 \varphi(z)\in\Oo_{1/\kappa}^{qk}\big(D\cup\bigcup_{l=0}^{q\kappa-1}S_{(d+\arg\lambda+2l\pi)/q}\big).
 \end{equation}
 Then $\hat{u}(t,z)$ is $k$-summable in the direction $d$ and its $k$-sum is given by
 \begin{gather}
  \label{eq:actual}
  u(t,z)=u^{d}(t,z)=
  \frac{1}{t^{1/q}}\int_{e^{\frac{i{d}}{q}}\RR_+}v(s^q,z)C_q(s/t^{1/q})\,ds,
 \end{gather}
 where 
\begin{equation}
\label{eq:v_est}
v(t,z):=\hat{\Bo}_{1/k}\hat{u}(t,z)=\hat{\Bo}_{1/k}(\sum_{n=0}^{\infty}\frac{u_n(z)}{n!}t^n)=\sum_{n=0}^{\infty}\frac{u_n(z)}{\Gamma(1+qn)}t^n
\in\Oo_{1,1/\kappa}^{q}((D\cup S_d)\times D)
\end{equation}
has the integral representation
\begin{equation}
 \label{eq:v_integral}
  v(t,z)=\frac{t^{\beta-1}}{(\beta-1)!}\partial_t^{\beta-1}
   \frac{1}{2\kappa\pi i}\oint_{|w|=\varepsilon}^{\kappa}\varphi(w)\int_{e^{i\theta}r_0}^{ e^{i\theta}\infty}\mathbf{E}_{q}(t\lambda(\zeta))
    \mathbf{E}_{1/\kappa}(\zeta^{1/\kappa} z^{1/\kappa})e^{-\zeta w}\,d\zeta\,dw
 \end{equation}
 with $\theta\in (-\arg w-\frac{\pi}{2}, -\arg w + \frac{\pi}{2})$.
 Moreover, if $\varphi\in\Oo_{\mathcal{A}}(\CC\setminus H)$ and $z\in D_{\tilde{\varepsilon}}$ for some $\tilde{\varepsilon}>0$
 then the function $t\mapsto v(t,z)$ is holomorphic for $|t|< \frac{(r-\tilde{\varepsilon})^q}{|\lambda|}$, where
 $r:=\min_{1\leq i \leq K}|a_{i1}|$.
\end{theorem}
\begin{proof}
 First, observe that by Proposition \ref{pr:family} we get $\hat{u}(t,z)\in \Oo_{1/\kappa}(D)[[t]]_{q-1}$.
 Moreover, by \cite[Proposition 7]{Mic8} the function 
$v(t,z)=\hat{\Bo}_{1/k}\hat{u}(t,z)\in\Oo_{1,1/\kappa}(D^2)$
 satisfies the moment partial differential equation
 \begin{equation}
 \label{eq:v}
  \begin{cases}
   (\partial_{t,\Gamma_q}-\lambda(\partial_z))^{\beta}v=0\\
   \partial_{t,\Gamma_q}^j v(0,z)=0\ (j=0,\dots,\beta-2)\\
  \partial_{t,\Gamma_q}^{\beta-1}v(0,z)=\lambda^{\beta-1}(\partial_z)\varphi(z)\in\Oo_{1/\kappa}(D),
  \end{cases}
 \end{equation}
where $\Gamma_q$ is a moment function defined by $\Gamma_q(n):=\Gamma(1+nq)$ for $n\in\NN_0$ and
$\partial_{t,\Gamma_q}$ is so called \emph{$\Gamma_q$-moment differential operator} defined by (see \cite[Definition 12]{Mic8})
$$\partial_{t,\Gamma_q}\Big(\sum_{n=0}^{\infty}\frac{a_n(z)}{\Gamma_q(n)}t^n\Big):=\sum_{n=0}^{\infty}\frac{a_{n+1}(z)}{\Gamma_q(n)}t^n.$$
Hence by \cite[Lemma 3]{Mic8} with $m_1(n)=\Gamma_q(n)$ and $m_2(n)=\Gamma(1+n)$ we get the integral representation (\ref{eq:v_integral})
of $v(t,z)$.

Since $\varphi(z)$ satisfies (\ref{eq:condition}),
by \cite[Lemma 4]{Mic8} we conclude that 
$v(t,z)\in\Oo_{1,1/\kappa}^{q}((D\cup S_d)\times D)$.
So, the function
$u^{d}(t,z):=\La_{k,d}v(t,z)$ is well-defined and is given by (\ref{eq:actual}).

Since the Mittag-Leffler function is the entire function satisfying the estimation $|\mathbf{E}_{q}(z)|\leq Ce^{|z|^{1/q}}$
(see \cite[Appendix B.4]{Balser}),
the integrand in the inner integral in (\ref{eq:v_integral}) is estimated for $|z|<\widetilde{\varepsilon}$ by
$$
|\mathbf{E}_{q}(t\lambda(\zeta))\mathbf{E}_{1/\kappa}(\zeta^{1/\kappa} z^{1/\kappa})e^{-\zeta w}|\leq 
\tilde{C} e^{|\zeta|(|\lambda|^{1/q}|t|^{1/q}-|w|+\widetilde{\varepsilon})}
$$
as $\zeta\to\infty$, $\arg\zeta= \theta = -\arg w$. By the hypothesis $\varphi(w)$ is holomorphic for $|w|<r$, so we may deform the path of
integration in the outer integral in (\ref{eq:v_integral}) from $|w|=\varepsilon$ to $|w|=\tilde{r}$ for any $\tilde{r}<r$.
It means that the inner integral
in (\ref{eq:v_integral}) is convergent for any $t$ satisfying $|t|<\frac{(r-\widetilde{\varepsilon})^q}{|\lambda|}$
and the function $t\mapsto v(t,z)$ is holomorphic for such $t$.
\end{proof}

\subsection{Hyperasymptotic expansion of solution of simple pseudodifferential equations}
Using the change of variables to (\ref{eq:actual}), as in the case of the heat equation we obtain
 $$ 
u^{\theta}(t,z)=\frac{1}{qt^{1/q}}\int_0^{e^{i\theta}\infty}\frac{1}{s^{1-\frac{1}{q}}}v(s,z)C_q((s/t)^{1/q})\,ds,
$$
so as $t\to 0$, $\arg t=\theta$ we conclude that
 $$
u^{\theta}(t,z)=\frac{1}{q|t|^{1/q}}\int_0^{\infty}\frac{1}{s^{1-\frac{1}{q}}}v(se^{i\theta},z)C_q((s/|t|)^{1/q})\,ds
$$
for any $\theta$ different from the Stokes lines, i.e. $\theta \neq q\lambda_i-\arg\lambda \mod 2\pi$ for $i=1,\dots,K$.

Now we are ready to repeat the construction of the hyperasymptotic expansion for the heat equation
under condition that $\varphi\in\Oo^{kq}_{\mathcal{A}}(\CC\setminus H)$ (i.e. $\varphi\in\Oo_{\mathcal{A}}(\CC\setminus H)$
and $\varphi(z)$ has the exponential growth of order $kq$ as $z\to\infty$, $z\in\CC\setminus H$).
We also assume that the direction $\theta$ is separated from the Stokes lines, i.e. that
$$\theta\in [0,2\pi)\setminus\bigcup_{i=1}^K(q\lambda_i - \arg \lambda-\delta, q\lambda_i - \arg \lambda+ \delta)\mod 2\pi\quad
\textrm{for fixed}\quad\delta>0.
$$

We put $f_0(s,z):=v(se^{i\theta},z)$, $r:=\min\limits_{1\leq i \leq K}|a_{i1}|-\widetilde{\varepsilon}$
and
$$\Omega_0(0,s):=\partial\Big( \{w\in\CC\colon |w|\leq  \frac{r^q}{|\lambda|} - \varepsilon\}\cup
\{w\in\CC\colon |w-s|\leq\frac{\varepsilon}{2}\}\Big)
$$ 
for some $\varepsilon\in \left(0,\frac{r^q}{2|\lambda|}\right)$. 

Observe that by Theorem \ref{th:sum} for any $z\in D_{\widetilde{\varepsilon}}$ the function $w\mapsto f_0(w,z)$
is holomorphic in the domain bounded by $\Omega_0(0,s)$. By \cite[Lemma 2]{Mic7}
$$
u(t,z)=\sum_{j=\beta-1}^{N_0-1}{j \choose \beta-1} \frac{\lambda^j(\partial_z)\varphi(z)}{j!}t^j + R_{N_0}(t,z).
$$
Moreover, as in the case of the heat equation
$$
R_{N_0}(t,z)=\frac{1}{q|t|^{1/q}}\int_0^{\infty}s^{N_0-1+\frac{1}{q}}C_q((s/|t|)^{1/q})f_1(s,z)\,ds,
$$
where $f_1(s,z)$ is defined as in (\ref{f_0_Taylor}).
\par
By (\ref{eq:v_est}) there exist positive constants $A'$ and $B'$ such that
\begin{gather*}
 |f_0(w,z)|\leq A' e^{B'|s|^k}\quad\textrm{for any}\quad w\in\Omega_0(0,s).
\end{gather*}
Hence
\begin{multline*}
|f_{1}(s,z)|\leq \frac{1}{2\pi}\int_{\Omega_{0}\left(0,\,s\right)}\frac{2A'e^{B'|s|^k}}{\varepsilon(\frac{r^q}{|\lambda|}-\varepsilon)^{N_{0}}}\,d|w|
\leq \frac{2A'e^{B'|s|^k}}{\varepsilon(\frac{r^q}{|\lambda|}-\varepsilon)^{N_{0}}}(\frac{r^q}{|\lambda|}-\varepsilon+\frac{\varepsilon}{2})\\
\leq
\frac{2A'r^qe^{B'|s|^k}}{|\lambda|\varepsilon(\frac{r^q}{|\lambda|}-\varepsilon)^{N_{0}}}=
\frac{A_0'e^{B'|s|^k}}{(\frac{r^q}{|\lambda|}-\varepsilon)^{N_{0}}},
\end{multline*}
where $A_0':=\frac{2A'r^q}{|\lambda|\varepsilon}$.

Moreover, by the properties of the Ecalle kernel (see \cite[Lemma 6]{MR}) we may estimate 
\begin{gather*}
 |C_{q}(\tau)|\leq Ce^{-(\tau^{k+1}/c_{q})} \quad \textrm{with} \quad c_{q}=(k+1)^{k+1}k^{-k}.
\end{gather*}

So
\begin{equation}
\label{eq:integrand}
|R_{N_0}(t,z)|\leq \frac{A_0'}{q|t|^{1/q}}
\int_0^{\infty}s^{N_0-1+\frac{1}{q}}e^{-s^k(\frac{1}{c_q|t|^k}-B')}(\frac{r^q}{|\lambda|}-\varepsilon)^{-N_0}\,ds.
\end{equation}
Similarly to the heat equation case we conclude that the integrand of (\ref{eq:integrand}) has a maximum at certain point $s=\sigma_1$
satisfying $N_0=k\sigma_1^k(\frac{1}{c_q|t|^k}-B') +1 - \frac{1}{q}$. Now, the minimum with respect to $\sigma_1$ is given at 
$\sigma_1= \frac{r^q}{|\lambda|} - \varepsilon$. Hence we take
$N_0:=\lfloor k(\frac{r^q}{|\lambda|} - \varepsilon)^k(\frac{1}{c_q|t|^k}-B') +1 - \frac{1}{q} \rfloor$ and
$\sigma_1:=\Big(\frac{N_0+\frac{1}{q}-1}{k(\frac{1}{c_q|t|^k}-B')}\Big)^{1/k}$. Observe that $\sigma_1\leq \frac{r^q}{|\lambda|} - \varepsilon$.
So we are able to use the Laplace method and to conclude that 
\begin{gather*}
|R_{N_0}(t,z)|\sim O\Big(\frac{e^{-\sigma_1^k(\frac{1}{c_q |t|^k}-B')}}{|t|^{1/q}\sqrt{\frac{1}{c_q |t|^k}-B'}}\Big)\quad\textrm{for}\quad
t\to 0,\ \arg t=\theta,\ z\in D_{\tilde{\varepsilon}}.
\end{gather*}

Next, we construct the $n$-level hyperasymptotic expansion as for the heat equation.
The remainder obtained in the $n$-level hyperasymptotic
expansion is of the form
\begin{equation}
\label{eq:RNn}
R_{N_{n}}(t,\,z)=\frac{1}{q|t|^{1/q}}\int_{0}^{\infty}C_q((\frac{s}{|t|})^{\frac{1}{q}})s^{N_{0}-1+\frac{1}{q}}(s-\sigma_{1})^{N_{1}}
\cdots(s-\sigma_{n})^{N_{n}}f_{n+1}(s,z)\,ds,
\end{equation}
where 
\[
f_{n+1}(s,z):=\frac{1}{2\pi i}\int_{\Omega_{n}\left(\sigma_{n},\,s\right)}\frac{f_{n}(w,z)}{(w-\sigma_{n})^{N_{n}}\left(w-s\right)}\,dw.
\]
Here we take
$$\Omega_{n}(\sigma_{n},s):=
\partial\Big(\{ w\in\CC\colon |w-\sigma_{n}|\leq d(\sigma_{n},\theta)-\rho_{n}\varepsilon\}\cup
\{ w\in\CC\colon |w-s|\leq 2^{-n-1}\varepsilon\}\Big)
$$
with $d(\sigma_n,\theta):=\inf_{z\in D_{\tilde{\varepsilon}}}\inf_{\zeta\in H}|\sigma_{n}-e^{-i\theta}\lambda(z-\zeta)^{q}|$
and $\rho_{n}:=2-2^{-n}$. 

Using the same algorithm as in the case of the heat equation, we
can estimate $R_{N_{n}}(t,z)$ as follows
\begin{multline}
\label{eq:RNn_est}
|R_{N_{n}}(t,z)|\leq \frac{A_{n}'}{q|t|^{1/q}}\int_{0}^{\infty}
e^{-s^k(\frac{1}{c_q|t|^k}-B')}\frac{s^{N_{0}-1+\frac{1}{q}}}{(\frac{r^q}{|\lambda|}-\varepsilon)^{N_{0}}}\cdot
\frac{|s-\sigma_{1}|^{N_{1}}}{\left(d(\sigma_{1},\theta)-\rho_{1}\varepsilon\right)^{N_{1}}}\cdot\\
\cdots\frac{|s-\sigma_{n}|^{N_{n}}}{\left(d(\sigma_{n},\theta)-\rho_{n}\varepsilon\right)^{N_{n}}}\,ds
\end{multline}
for a certain constant $A_{n}'$.
\par
Let us observe that the integrand on the right-hand side of
(\ref{eq:RNn}) has $n+1$ maxima in points $s_{1},\ldots,\,s_{n+1}$ such
that $s_{1}<\sigma_{1}<s_{2}<\ldots<\sigma_{n}<s_{n+1}$ and all $s_{j}$
satisfy the condition:
\begin{equation*}
k s_j^{k-1}\left(\frac{1}{c_q|t|^k}-B'\right)
=\frac{N_{0}-1+\frac{1}{q}}{s_{j}}+\frac{N_{1}}{s_{j}-\sigma_{1}}+\ldots+\frac{N_{n}}{s_{j}-\sigma_{n}}.
\end{equation*}
From this, as in the case of the heat equation, we conclude that $s_{j}$ are decreasing functions of $N_{n}$
for $1\leq j\leq n$ and $s_{n+1}$ increases to infinity as $N_{n}\to\infty$.
Similarly, the value of the integrand in (\ref{eq:RNn_est}) in the points $s_j$, $1\leq j\leq n$, decreases with respect to $N_{n}$. 
Moreover, this value in the point $s_{n+1}$ decreases with respect to $N_{n}$ for $s_{n+1}<\sigma_{n}+d(\sigma_{n},\theta)-\rho_{n}\varepsilon$
and increases when $s_{n+1}>\sigma_{n}+d(\sigma_{n},\theta)-\rho_{n}\varepsilon$.

Hence, as in the case of the heat equation there exists $N_{n}\in\mathbb{N}$ and $s_{n+1}$ for which the integrand reaches its minimal value.
We denote such $s_{n+1}$ by $\sigma_{n+1}$.

Again, using the Laplace method we obtain the estimation of $R_{N_n}$
\begin{gather*}
|R_{N_n}(t,z)|\sim O\Big(\frac{e^{-\tilde{\eta}_n^k(\frac{1}{c_q |t|^k}-B')}}{|t|^{1/q}\sqrt{\frac{1}{c_q |t|^k}-B'}}\Big)\quad\textrm{for}\quad
t\to 0,\ \arg t=\theta,\ z\in D_{\tilde{\varepsilon}},
\end{gather*}
where, as previously $\tilde{\eta}_0=\sigma_1\sim \frac{r^q}{|\lambda|}<\tilde{\eta}_1<\tilde{\eta}_2<\tilde{\eta}_3<\dots$ is some increasing
sequence of positive numbers (see \cite{O1a} and \cite{O1b}).

To find the $(n+1)$-level hyperasymptotic expansion we expand the function $s\mapsto f_{n+1}(s,z)$
around the point $\sigma_{n+1}$ as in (\ref{eq:f_n1}),
which, after substituting it in (\ref{eq:RNn}), gives us the ($n+1$)-level
expansion
\begin{multline*}
R_{N_{n}}(t,z)=\frac{1}{q|t|^{1/q}}\sum_{j=0}^{N_{n+1}-1}b_{n+1,j}(z)\int_{0}^{\infty}
C_q((\frac{s}{|t|})^{\frac{1}{q}})s^{N_{0}-1+\frac{1}{q}}(s-\sigma_{1})^{N_{1}}\\
\cdots(s-\sigma_{n})^{N_{n}}(s-\sigma_{n+1})^j\,ds+R_{N_{n+1}}(t,z).
\end{multline*}

Since the Laplace transform $\La_{k,d}$ is inverse to $k$-Borel transform $\hat{\Bo}_k$, we conclude that
$\La_{k,d}(t^l)=\frac{\Gamma(1+ql)}{l!}t^l$ for $l=0,1,\dots$.
It means that
\begin{gather*}
 \frac{1}{q|t|^{1/q}}\int_0^{\infty}\frac{C_q((\frac{s}{|t|})^{\frac{1}{q}})}{s^{1-\frac{1}{q}}}s^l\,ds=\frac{\Gamma(1+ql)}{l!}|t|^l
 =\frac{\Gamma(1+ql)}{l!}e^{-i\theta l}t^l,
\end{gather*}
and using (\ref{eq:polynomial}) we get
$$
R_{N_{n}}(t,z)=\sum_{j=0}^{N_{n+1}-1}b_{n+1,j}(z)\sum_{l=0}^{N_0+\cdots+N_n+j}a_{n,j,l}\frac{\Gamma(1+ql)}{l!}e^{-i\theta l}t^l+
R_{N_{n+1}}(t,z).
$$

Hence, as in the case of the heat equation, we conclude that the hyperasymptotic expansion of $u$ takes also the form (\ref{eq:exp_1})
for some functions $\psi_l(z)$.

Finally, similarly to the heat equation, we get as the conclusion

 \begin{theorem}[Hyperasymptotic expansion for the simple equation]
   For every $n\in\NN$ the solution of the equation (\ref{eq:simple}) with $\varphi\in\Oo^{kq}_{\mathcal{A}}(\CC\setminus H)$
   has the hyperasymptotic expansion as $t$ tends to zero
   in a direction $\theta\in [0,2\pi)\setminus\bigcup_{i=1}^K(q\lambda_i - \arg \lambda-\delta, q\lambda_i - \arg \lambda+ \delta)\mod 2\pi$,
   which has the form
  \begin{multline*}
   u^{\theta}(t,z)=\sum_{j=\beta-1}^{N_0-1}{j \choose \beta-1}\frac{\lambda^j(\partial_z)\varphi(z)}{j!}t^j
   +\sum_{m=1}^n\sum_{j=0}^{N_m-1}
   \frac{b_{m,j}(z)}{|t|^{1/q}}
   \int_0^{\infty}\frac{1}{qs^{1-\frac{1}{q}}}\\
   \cdot C_q((s/|t|)^{1/q})s^{N_0}(s-\sigma_1)^{N_1}\cdots (s-\sigma_{m-1})^{N_{m-1}}(s-\sigma_m)^j\,ds+
   R_{N_n}(t,z)\\
   =\sum_{l=0}^{N_0+\cdots+N_n-1}\psi_l(z)t^l+R_{N_{n}}(t,z),
   \end{multline*}
  where 
  $$b_{m,j}(z)=\frac{1}{j!}\frac{\partial^j}{\partial s^j}f_m(s,z)|_{s=\sigma_m},$$
 \begin{multline*}
  R_{N_n}(t,z)=\frac{1}{|t|^{1/q}}\int_0^{\infty}\frac{1}{qs^{1-\frac{1}{q}}}C_q((\frac{s}{|t|})^{\frac{1}{q}})s^{N_0}(s-\sigma_1)^{N_1}\cdots
 (s-\sigma_{n})^{N_{n}}f_{n+1}(s,z)\,ds,
  \end{multline*}
 \begin{multline*}
 f_m(s,z)=\frac{1}{(2\pi i)^m}\int_{\Omega_{m-1}(\sigma_{m-1},s)}\int_{\Omega_{m-2}(\sigma_{m-2},x_{m-1})}\cdots
 \int_{\Omega_{1}(\sigma_{1},x_2)}\int_{\Omega_{0}(0,x_1)}\\
   \frac{v(x_0e^{i\theta},z)\,dx_0\dots dx_{m-1}}{x_0^{N_0}\big[\prod_{k=1}^{m-1}(x_k-\sigma_k)^{N_k}(x_{k-1}-x_{k})\big](x_{m-1}-s)}
 \end{multline*}
 and $v(s,z)$ is defined by (\ref{eq:v_integral}).
 \par
 Moreover $R_{N_n}(t,z)\sim O\Big(\frac{e^{-\tilde{\eta}_n^k(\frac{1}{c_q |t|^k}-B')}}{|t|^{1/q}\sqrt{\frac{1}{c_q |t|^k}-B'}}\Big)$
 as $t\to 0$, $\arg t = \theta$, $z\in D_{\tilde{\varepsilon}}$
 for some sequence of positive numbers $\tilde{\eta}_0=\sigma_1\sim \frac{r^q}{|\lambda|}<\tilde{\eta}_1<\tilde{\eta}_2<\tilde{\eta}_3<\dots$.
 \end{theorem}

 \subsection*{Acknowledgements}
The authors would like to thank the anonymous referee for valuable comments, suggestions, and especially for the indication of the
form of hyperasymptotic expansion of the solution $u(t,z)$ presented in (\ref{eq:exp_1}).

\end{document}